\documentclass[12pt,a4paper]{amsart}
\usepackage{amsfonts}
\numberwithin{equation}{section}

\theoremstyle{plain}
\newtheorem{Th}{Theorem}[section]
\newtheorem{Lemma}[Th]{Lemma}

 \theoremstyle{definition}

\newtheorem{?}[Th]{Problem}

\newcommand{\ab}[1]{\left\vert{#1}\right\vert}
\newcommand{\zj}[1]{\left({#1}\right)}

\newcommand{\Z}{\ensuremath{\mathbb{Z}}}

     \newcommand\pg{\overset{G}{+}}

\def \< {\langle}
\def \> {\rangle}

\begin{document}

\title[Superadditivity and submultiplicativity for sumsets]{A superadditivity and submultiplicativity property for cardinalities of
sumsets}

\author{Katalin Gyarmati}
\address{Alfr\'ed R\'enyi Institute of Mathematics\\
     Budapest, Pf. 127\\
     H-1364 Hungary
}
\email{gykati@cs.elte.hu}
\thanks{Supported by Hungarian National Foundation for Scientific Research
(OTKA), Grants No. T 43631, T 43623, T 49693. }

\author{M\'at\'e Matolcsi}
\address{Alfr\'ed R\'enyi Institute of Mathematics\\
     Budapest, Pf. 127\\
     H-1364 Hungary
} \email{matomate@renyi.hu}
\thanks{Supported by Hungarian National Foundation for Scientific Research
(OTKA), Grants No. PF-64061, T-049301, T-047276}

\author{Imre Z. Ruzsa}
\address{Alfr\'ed R\'enyi Institute of Mathematics\\
     Budapest, Pf. 127\\
     H-1364 Hungary
}
\email{ruzsa@renyi.hu}
\email{{\rm To all authors: } triola@renyi.hu}
\thanks{Supported by Hungarian National Foundation for Scientific Research
(OTKA), Grants No. T 43623, T 42750, K 61908.}

    \begin{abstract}
     For finite sets of integers $A_1, A_2 \dots A_n$  we study the cardinality of the $n$-fold sumset $A_1+\dots
     +A_n$ compared to those of $n-1$-fold sumsets $A_1+\dots
     +A_{i-1}+A_{i+1}+\dots A_n$. We prove a superadditivity and a
     submultiplicativity property for these quantities. We also
     examine the case when the addition of elements is restricted
     to an addition graph between the sets.
    \end{abstract}

     \maketitle

     \section{Introduction}

Let $A_1, A_2, \dots A_n$ be finite sets of integers. How does the
cardinality of the $n$-fold sumset $A_1+ A_2 +\dots +A_n$ compare
to the cardinalities of the $n-1$-fold sums $A_1+\dots
     +A_{i-1}+A_{i+1}+\dots A_n$?

In the special case when all the sets are the same, $A_i=A\subset
\Z$, Vsevolod Lev \cite{lev96} observed that the quantity $|kA|-1 \over k$
     is increasing (where we have used the standard notation for the $k$-fold sum $A+A+\dots +A=kA$). The first cases of this result assert that
     \begin{equation} \label{12}
     |2A| \geq  2|A| -1
     \end{equation}
     and
     \begin{equation} \label{23}
     |3A| \geq  {3\over 2}|2A| - {1\over 2}.
     \end{equation}
     Inequality \eqref  {12} can be extended to different summands as
     \begin{equation} \label{2d}
     |A+B| \geq  |A| +|B| -1,
     \end{equation}
     and this inequality also holds for sets of residues modulo a prime
$p$, the only obstruction being that a cardinality cannot exceed
$p$, i.e.
     \begin{equation} \label{2p}
     |A+B| \geq  \min (|A| +|B| -1, p) ;
     \end{equation}
     this familiar result is known as the Cauchy-Davenport inequality.

     The third author asked whether inequality \eqref  {23} can also be extended
to different summands in the following form:
       \begin{equation} \label{harom} |A+B+C| \geq  { |A+B| + |B+C| + |A+C| -1 \over 2} . \end{equation}
     Lev noticed (personal communication) that this is true in the case when
the sets have the same diameter. (The diameter of a set is the
difference of its maximum and minimum.) In this paper we establish
this property in general, for an arbitrary number of summands, and
with the extra twist that in the $n$-fold sumset it is sufficient
to use the smallest or largest element of at least one of the
summands.

     \begin{Th}  \label{superadd}
     Let $A_1, \dots , A_k$ be finite, nonempty sets of integers. Let $A_i'$ be the two-
(or possibly one-) element set containing the smallest and largest
elements of $A_i$. Put
     \[   S = A_1 + \dots  + A_k ,  \]
     \[   S_i = A_1 + \dots  + A_{i-1} + A_{i+1}+ \dots  + A_k ,  \]
     \[   S_i' = A_1 + \dots  + A_{i-1} + A_i' + A_{i+1}+ \dots  + A_k , \]
     \[   S' = \bigcup _{i=1}^k S_i' .  \]
     We have
     \begin{equation} \label{also}
     \left| S \right| \geq  \left|S' \right| \geq  {1\over k-1} \sum _{i=1}^k \left|S_i \right|
- {1\over k-1}.
     \end{equation}
     \end{Th}

     The possibility to extend inequality \eqref  {23}  to residues modulo a prime $p$ was
investigated in a paper by Gyarmati, Konyagin, Ruzsa \cite{r07e}.
A naive attempt to extend it in the form
     \[   |3A| \geq  \min \left( {3\over 2}|2A| - {1\over 2}, p \right)  \]
     holds only when $\left|A \right|$ is small in comparison to $p$,
and for larger values the relationship between the sizes of $2A$ and $3A$ is
complicated.

     In a sense, Theorem \ref{superadd} means that the cardinality of sumsets grows faster than
linear. On the other hand, we show that it grows slower than
exponential. For identical summands this means that $\left|kA
\right|^{1/k}$ is decreasing. This was conjectured by the third author.
Lev observed that this is a straightforward consequence of a
Pl\"unnecke-type inequality; more details will be given in Section 4.

Here we establish a more general result for different summands.

     \begin{Th}  \label{submult}
     Let $A_1, \dots , A_k$ be finite, nonempty sets in an arbitrary commutative
semigroup. Put
     \[   S = A_1 + \dots  + A_k ,  \]
     \[   S_i = A_1 + \dots  + A_{i-1} + A_{i+1}+ \dots  + A_k .  \]
     We have
     \begin{equation} \label{felso}
     \left| S \right| \leq  \left( \prod _{i=1}^k \left|S_i \right|  \right)^ {1\over k-1}.
     \end{equation}
     \end{Th}

     For three summands this inequality was established earlier by the third author
(\cite{r07c}, Theorem 5.1). The proof given in \cite{r07c} is different and
works also for noncommutative groups with a proper change in the
formulation. On the other hand, that argument relied on the invertibility of
the operation, so we do not have any result for noncommutative semigroups.
Neither could we  extend that argument
for more than three summands, and hence the following question remains open.

     \begin{?}  \label{nemkommutativ}
     Let $A_1, \dots , A_k$ be finite, nonempty sets in an arbitrary noncommutative
group. Put
     \[   S = A_1 + \dots  + A_k ,  \]
     \[   n_i = \max _{a\in A_i} \left| A_1 + \dots  + A_{i-1} + a+ A_{i+1}+ \dots
+ A_k \right| . \]
     Is it true that
     \begin{equation} \label{nemkom}   \left| S \right| \leq  \left( \prod _{i=1}^k n_i  \right)^ {1\over k-1}  \ ?\end{equation}
     \end{?}

     The superadditivity property clearly does not hold in such a general
setting (as it fails already mod $p$, see \cite{r07e}). However,
it can easily be extended to torsion-free groups (just as
everything that holds for finite sets of integers) with the change
of formulation that ``smallest'' and ``largest'' do not make sense
in such generality.

     \begin{Th}  \label{superadd-tf}
     Let $A_1, \dots , A_k$ be finite, nonempty sets in a torsion-free group $G$,
     \[   S = A_1 + \dots  + A_k ,  \]
     \[   S_i = A_1 + \dots  + A_{i-1} + A_{i+1}+ \dots  + A_k .  \]
      There are subsets
$A_i'\subset A_i$ having at most two elements such that with
     \[   S_i' = A_1 + \dots  + A_{i-1} + A_i' + A_{i+1}+ \dots  + A_k , \]
     \[   S' = \bigcup _{i=1}^k S_i'   \]
     we have
     \begin{equation} \label{also-tf}
     \left| S \right| \geq  \left|S' \right| \geq  {1\over k-1} \sum _{i=1}^k \left|S_i \right|
- {1\over k-1}.
     \end{equation}
     \end{Th}

Another natural way of generalizing Theorem \ref{submult} is to
restrict the summation of elements to a prescribed addition graph.
A possible meaning of this in the case $k=3$ (and identical sets)
could read as follows. We consider a graph $G$ on our set $A$; on
the right hand side of the proposed inequality we take the number
of different sums of connected pairs; on the left hand side we
take the number of different sums of those triplets where each
pair is connected. However, the resulting
 inequality, $|A+\pg  A+\pg  A|^2\leq |A+\pg  A|^3$, can fail
spectacularly. Take  $A=[1,n]$, let $S\subset (2n/3, 4n/3)$ be a set of even integers
 and connect two elements of $A$ if their sum is in $S$. Then for every
  $s_1, s_2, s_3\in S$ we can find $a_1, a_2, a_3\in A$, $a_1=(-s_1+s_2+s_3)/2$, etc., whose pairwise sums give these
$s_i$'s. Also,  $a_1+ a_2+ a_3 = (s_1+ s_2+ s_3)/2 $. Therefore,
if $S$ is such that all the triple sums $s_1+s_2+s_3$ are
distinct, then the above mapping $(s_1, s_2, s_3)\mapsto
(a_1,a_2,a_3)$ is injective, and the left side of the inequality
will be $\binom{|S|}{3}^2\approx \frac{1}{6}|S|^6$, much larger
than the right hand side, which is $S^3$.

 It would
be interesting to say something when the graphs are sufficiently dense.

     However, we will prove a similar statement in the case when only one pair
of summands is restricted.

\begin{Th}\label{3sum}
Let $A,B_1,B_2$ be finite sets in a commutative group, and $S\subset B_1+
B_2$. Then
\begin{equation} \label{restsum}
|S+A|^2\leq |S||A+B_1||A+B_2|
\end{equation}
\end{Th}
The analogous statement for more than three sets remains an open problem.
\begin{?}
Let $A,B_1, \dots B_k$ be finite sets of integers, and $S\subset
B_1+ \dots +B_k$. Is it true that
\begin{equation} \label{altrestsum}
|S+A|^{k}\leq |S|\prod_{i=1}^k|A+B_{1} + \dots  + B_{i-1} + B_{i+1}+ \dots  + B_{k} |\ ?
\end{equation}
\end{?}

     \section{Proof of superadditivity}\label{sec1}
In this section we prove Theorems \ref{superadd} and
\ref{superadd-tf}.

     \begin{proof}[Proof of Theorem \ref{superadd}.]

Both sides of the inequality are invariant under translation,
therefore we can assume that the smallest element of each $A_i$ is
0. Also, let us denote the largest element of $A_i$ by $a_i$. Then
$S$ is a subset of the interval $[0, a_1+a_2+\dots + a_k].$

Make $k-1$ copies of the set $S$. In the first copy mark the
elements of \\
$0+A_1+\dots +A_{k-1}$. They all belong to the interval $[0,
a_1+\dots +a_{k-1}].$ In the remaining interval $(a_1+\dots
+a_{k-1}, a_1+\dots +a_{k-1}+a_k]$ of the first copy of $S$ mark
the elements of $a_{k-1}+A_1+A_2+\dots +A_{k-2}+A_k$ which fall in
there. These elements correspond exactly to the elements of
$A_1+A_2+\dots +A_{k-2}+A_k$ which are larger than $a_1+\dots
+a_{k-2}$. We denote this latter set by $(A_1+A_2+\dots
+A_{k-2}+A_k)_{>a_1+\dots +a_{k-2}}$.

Then, for $2\leq i \leq k-2,$ in the $i$th copy of $S$ mark the
elements of \\
$0+(A_1+A_2+\dots +A_{k-i}+A_{k-i+2}+\dots +A_k)_{\leq a_1+\dots
+a_{k-i}}$, and the elements of\\
 $a_{k-i}+(A_1+\dots
+A_{k-i-1}+A_{k-i+1}+\dots +A_k)_{>a_1+\dots a_{k-i-1}}.$ Finally,
in the $k-1$st copy of $S$ mark the elements of $0+(A_1+A_3+\dots
+A_k)_{\leq a_1}$ and the elements of $a_1+(A_2+\dots
+A_k)_{>a_1}.$

Note that all marked elements belong to $S'$. Also, for $1\leq i\leq
k-2$ the number of marked elements in the second section of the
$i$th copy and the first section of the $i+1$st copy of $S$ is
exactly $|A_1+\dots +A_{k-i-1}+A_{k-i+1}+\dots +A_k|$.
Furthermore, the number of marked elements in the first section of
the first copy is $|A_1+\dots +A_{k-1}|$, while in the second
section of the last copy it is $|A_2+A_3+\dots +A_k|-1$. Let $M$
denote the set of marked elements. Then, by construction,
\begin{equation} \label{m}(k-1)|S|\geq (k-1)|S'|\geq |M|= \sum _{i=1}^k \left|S_i
\right|-1\end{equation} and we are done.
     \end{proof}

          \begin{proof}[Proof of Theorem \ref{superadd-tf}.]

This is a standard reduction argument to the case of integers. Let
$H$ denote the subgroup generated by the elements of $\cup_{i=1}^k
A_i.$ As a finitely generated torsion-free group $H$ is isomorphic
to $\Z^d$ for some $d$, therefore we can assume without loss of
generality that $A_i\subset \Z^d.$ Then, for a large enough
integer $m$ the homomorphism $\phi_m : \Z^d\to \Z$ defined by
$(z_1, z_2, \dots z_d)\mapsto mz_1+m^2z_2+\dots m^dz_d$ preserves
the additive identities of all elements of sumsets involved in the
desired inequality (this means that $\phi_m$ is one-to-one
restricted to these elements). Finally, if $B_i$ denotes the image
of $A_i$ under $\phi_m$ then the desired two-element subsets
$A_i'$ can be chosen as $A_i'=\phi_m^{-1}(B_i')$.
     \end{proof}

     \section{Proof of submultiplicativity}\label{sec2}

In this section we prove Theorem \ref{submult}. We begin with a
lemma on the size of projections.

     \begin{Lemma}  \label{vetulet}
     Let $d\geq 2$ be an integer, $X_1, \dots , X_d$ arbitrary sets,
     \[   B \subset  X_1 \times  \dots  \times  X_d  \]
     be a finite subset of their Cartesian product. Let
     \[   B_i \subset  X_1 \times  \dots  \times  X_{i-1} \times  X_{i+1} \times  \dots  \times  X_d  \]
     be the corresponding ``projection'' of $B$:
     \[   B_i = \{ (x_1, \dots , x_{i-1}, x_{i+1}, \dots , x_d): \exists x\in
     X_i \
\text{such \ that} \  (x_1, \dots , x_{i-1}, x,  x_{i+1}, \dots ,
x_d)\in B      \}. \]
     We have
     \begin{equation} \label{vetuletbecsles}
     \left|B \right|^{d-1} \leq  \prod _{i=1}^d \left|B_i \right|.
     \end{equation}
     \end{Lemma}

     This lemma is not new. It is essentially equivalent to an
entropy inequality of Han \cite{han78}, see also Cover--Thomas
\cite{coverthomas91}, Theorem 16.5.1. It follows from Shearer's inequality
\cite{chungetal86} or from Bollob\'as and Thomason's Box Theorem
\cite{bollobasthomason95}. We include a proof for fun.

     \begin{proof} % \label{ }

We prove this lemma by induction on $d$. For $d=2$ the
statement is obvious. Assume now that the statement holds for
$d-1$, and consider the case $d$.

Make a list $\{b_1,b_2,\dots,b_t\}$ of those elements of $X_1$
which appear as a first coordinate of some element in $B$.
Partition the set $B$ according to these first coordinates as
\begin{equation} \label{b}
 B=B(b_1)\cup B(b_2)\cup\dots\cup B(b_t),
\end{equation}
where
\begin{equation} \label{bbi}
 B(b_i)=\{(b_i,x_2,x_3,\dots,x_d)=b:\ b\in B\}.
\end{equation}
By the inductive hypothesis we have $\ab{B(b_i)}^{d-2}\leq
\ab{B(b_i)_2}\cdots \ab{B(b_i)_{d}},$ that is,
\begin{equation} \label{absbbi}\ab{B(b_i)}^{\frac{d-2}{d-1}}\leq \zj{\ab{B(b_i)_{2}}\cdots
\ab{B(b_i)_{d}}}^{\frac{1}{d-1}}.\end{equation} It is also clear
that $\ab{B(b_i)}\leq \ab{B_{1}},$ and hence
\begin{equation} \label{bbiuj}
\ab{B(b_i)}\leq \zj{\ab{B(b_i)_{2}}\cdots
\ab{B(b_i)_{d}}}^{\frac{1}{d-1}} \ab{B_{1}}^{\frac{1}{d-1}}.
\end{equation}
Using this and H\"older's inequality we obtain
\begin{eqnarray}\label{abb}
\ab{B}= \sum_{i=1}^{t} \ab{B(b_i)}\leq
\ab{B_{1}}^{\frac{1}{d-1}}\sum_{i=1}^{t} \zj{\ab{B(b_i)_{2}}\cdots
\ab{B(b_i)_{d}}}^{\frac{1}{d-1}}\leq \\
 \leq \ab{B_{1}}^{\frac{1}{d-1}}
\prod_{j=2}^{d}
\zj{\sum_{i=1}^{t}\ab{B(b_i)_{j}}}^{\frac{1}{d-1}}=
\prod_{j=1}^{d}\ab{B_{j}}^{\frac{1}{d-1}},
\end{eqnarray} which
proves the statement.
     \end{proof}

We now turn to the  proof of Theorem \ref{submult}.

     \begin{proof} % \label{ }

Let us list the elements of the sets $A_1, A_2,\dots, A_k$ in some
order:
 % $A_1=\{c_{11},c_{12},\dots,c_{1t_1}\},$\\
 % $A_2=\{c_{21},c_{22},\dots,c_{2t_2}\},$\\
 % $\vdots$\\
 % $A_k=\{c_{k1},c_{k2},\dots,c_{kt_k}\}.$
\[A_1=\{c_{11},c_{12},\dots,c_{1t_1}\},\]
\[A_2=\{c_{21},c_{22},\dots,c_{2t_2}\},\]
\[\vdots\]
\[A_k=\{c_{k1},c_{k2},\dots,c_{kt_k}\}.\]
For each $s\in S$ let us consider the decomposition
\begin{equation} \label{sdecomp}
 s=c_{1i_1}+c_{2i_2}+\dots+c_{ki_k},
\end{equation}
where the finite sequence $(i_1,i_2,\dots,i_k),$ composed of the
(second) indices of $c_{ji_j}$, is minimal in lexicographical
order.  Let us define a function $f$ from $S$ to the Cartesian
product $A_1\times A_2\times \cdots \times A_k$, by
\begin{equation} \label{fdef}
f(s)=(c_{1i_1},c_{2i_2},\dots,c_{ki_k})\in A_1\times\cdots\times
A_k.
\end{equation}
This function is well-defined, and it maps the set $S$ to a set
$B\subset A_1\times\cdots\times A_k$ such that
$\ab{B}=\ab{A_1+\dots+A_k}.$ Applying Lemma \ref{vetulet} to the
set $B$ we get
\begin{equation} \label{bappl}
\ab{B}^{k-1}\le\ab{B_{1}}\ab{B_{2}}\cdots \ab{B_{k}}.
\end{equation}
Therefore, it is sufficient to show that
\begin{equation} \label{bjaj}\ab{B_{j}}\leq
\ab{A_1+A_2+\dots+A_{j-1}+A_{j+1}+\dots+A_k}.\end{equation} This
inequality, however, follows easily from the fact that sum of the
coordinates is distinct for each element in $B_{j}$. Indeed,
assume that there exist two elements $z\ne z'\in B_{j}$ such that
\[z=(c_{1i_1},c_{2i_2},\dots,c_{j-1i_{j-1}},c_{j+1i_{j+1}},\dots c_{ki_k}),\]
\[z'=(c_{1i_1'},c_{2i_2'},\dots,c_{j-1i_{j-1}'},c_{j+1i_{j+1}'},\dots
,c_{ki_k'}), \]
 % $z=(c_{1i_1},c_{2i_2},\dots,c_{j-1i_{j-1}},c_{j+1i_{j+1}},\dots
 % c_{ki_k}),$\\
 % $z'=(c_{1i_1'},c_{2i_2'},\dots,c_{j-1i_{j-1}'},c_{j+1i_{j+1}'},\dots
 % ,c_{ki_k'})$,
 and
\[
c_{1i_1}+c_{2i_2}+\dots+c_{ki_k}=c_{1i_1'}+c_{2i_2'}+\dots+c_{ki_k'}.
\]
We may assume that
\[
(i_1,i_2,\dots,i_{j-1},i_{j+1},\dots,i_k)
<(i_1',i_2',\dots,i_{j-1}',i_{j+1}',\dots,i_k').
\] in lexicographical order.

Now, $z'\in B_{j}$ therefore there exists an element $d\in A_j$
and $u\in S$, such that
\[
u=c_{1i_1'}+c_{2i_2'}+\dots+c_{j-1i_{j-1}'}+d+c_{j+1i_{j+1}'}+\dots+c_{ki_k'},
\]
and
\[
f(u)=(c_{1i_1'},c_{2i_2'},\dots,c_{j-1i_{j-1}'},d,c_{j+1i_{j+1}'},
\dots ,c_{ki_k'}) \in B.
\]
Note that
\[
u=c_{1i_1}+c_{2i_2}+\dots+c_{j-1i_{j-1}}+d+c_{j+1i_{j+1}}+\dots+c_{ki_k},
\]
also holds. However, with $d=c_{ji_j}$ we have
\[
(i_1,i_2,\dots,i_{j-1},i_j,i_{j+1},\dots,i_k) <
(i_1',i_2',\dots,i_{j-1}',i_j,i_{j+1}',\dots,i_k').
\]
in lexicographical order, therefore the definition of $f$ implies
that \\
$f(u)\ne (c_{1i_1'},c_{2i_2'},\dots ,
c_{j-1i_{j-1}'},d,c_{j+1i_{j+1}'},\dots ,c_{ki_k'})$, a
contradiction.
     \end{proof}

     A similar method is used by Alon \cite{alon2003} for the particular case
when we have sets instead of numbers, the operation is intersection, and the
sets $A_i$ are identical. As Alon observes, the same approach works for general
semigroups where the elements are idempotent.

\section{Restricted sums and Pl\"unnecke-type results}

     Pl\"unnecke \cite{plunnecke70} developed a graph-theoretic method to
estimate the density of sumsets $A+B$, where $A$ has a
positive density and $B$ is a basis. The third author published a simplified
version of his
proof \cite{r89e,r90a}. Accounts of this method can be found in
  Malouf
\cite{malouf95}, Nathanson \cite{nathanson96}, Tao and Vu \cite{taovu06}.

     The simplest instance of Pl\"unnecke's inequality for finite sets goes as
follows.

     \begin{Th} \label{plunnalk}
Let $i<k$ be integers, $A$, $B$ sets in a commutative group and write
$|A|=m$, $|A+iB|=\alpha m$. There is an $X \subset  A$, $X  \ne   \emptyset $ such that
     \begin{equation} \label{pl1}
      |X+kB| \leq \alpha ^{k/i} |X|.
     \end{equation}
     \end{Th}

     As Lev observed, this is sufficient to deduce the monotonicity of $\left|
kA\right|^{1/k}$. Indeed, in the above result replace $B$ by $A$ and $A$ by
$\{0\}$. Then $\alpha = \left|iA \right|$, the only possibility is $X=\{0\}$ and \eqref
{pl1} reduces to $\left|kA \right|\leq \left|iA \right|^{k/i}$.

     The application to different summands is less straightforward. We start
from the following result from \cite{r89e}, which extends the case $i=1$ of
Theorem \ref{plunnalk} to the addition of different sets.

     \begin{Th} \label{pldiff}
Let $A$, $B_1, \dots  , B_h$ be finite sets in a commutative group and write
$|A|=m$,\\
$|A+B_i|=\alpha _i m$, for $1\leq i\leq h$. There exists an $X
\subset A$, $X \ne \emptyset $ such that
\begin{equation}\label{difplu}
|X+B_1+ \dots +B_h| \leq \alpha _1 \alpha _2 \dots  \alpha _h |X| .
\end{equation}
     \end{Th}

In the sequel we will need a 'large' subset
$X\subset A$, not just a non-empty one. This will be achieved by the
following result.

\begin{Th} \label{pldiffnagy}
Let $A$, $B_1, \dots  , B_h$ be finite sets in a commutative group and write
$|A|=m$,\\
 $\prod |A+B_i|= s$, $B_1 + \dots  + B_h = B$.
 Let an integer $k$ be
given, $1\leq k\leq m$.  There is an $X \subset  A$, $|X|\geq k$
such that
     \begin{equation} \label{pldn}
     \left|X+B \right|  \leq  {s\over m^h} + {s\over (m-1)^h} + \dots  + {s\over (m-k+1)^h} +(|X|-k) {s\over (m-k+1)^h}.
     \end{equation}
     \end{Th}

     \begin{proof}
     We use induction on $k$. The case $k=1$ is Theorem \ref{pldiff}.

     Assume we know it for  $k$; we prove it for $k+1$. The assumption gives us
a set $X$, $|X|\geq k$ with a bound on $|X+B| $ as given by
\eqref{pldn}. We want to find a set $X'$ with $|X'|\geq k+1$ and
     \begin{equation} \label{pld1}
     \left|X'+B \right|  \leq  {s\over m^h} + {s\over (m-1)^h} + \dots  + {s\over (m-k)^h} +(|X'|-k-1) {s\over (m-k)^h}.
     \end{equation}
 If $|X|\geq k+1$, we can put $X'=X$. If $|X|=k$, we apply Theorem \ref{pldiff}
to the sets $A\setminus X$, $B_1$, \dots , $B_h$.
 This yields a
set $Y\subset A\setminus X$ such that
     $$   |Y+B| \leq  {s\over (m-k)^h} |Y|  $$
     and we put $X'=X\cup Y$.
     \end{proof}

     The following variant will be more comfortable for calculations.

     \begin{Th} \label{pldiffnagy2}
Let $A$, $B_1, \dots  , B_h$ be finite sets in a commutative group and write
$|A|=m$, \\
$\prod |A+B_i|= s$, $B_1 + \dots  + B_h = B$.
 Let a real number $t$ be given, $0\leq t<m$.
 There is an $X \subset  A$, $|X|>t$  such that
     \begin{equation} \label{pld2}
     |X+B| \leq
     {s\over h-1} \left( {1\over (m-t)^{h-1}} - {1\over m^{h-1}} \right) + (|X|-t) {s\over (m-t)^h}.
     \end{equation}
     \end{Th}

     \begin{proof}
     We apply Theorem \ref{pldiffnagy} with $k=[t]+1$. The right side of \eqref
{pld2} can be written as
     $ s \int _0^{|X|} f(x) \,d x$,
     where $ f(x) = (m-x)^{-h}$ for $0\leq x\leq t$, and  $ f(x) = (m-t)^{-h}$ for
$t<x\leq |X|$. Since $f$ is increasing, the integral is
     $\geq  f(0)+f(1)+ \dots  + f(|X|-1) $.
     This exceeds the right side of \eqref{pldn} by a termwise comparison.
     \end{proof}

     \begin{proof}[Proof of Theorem \ref{3sum}.]

      Let us use the notation
$|A|=m$, $s=|A+B_1||A+B_2|$, as above.

Observe that if $|S|\leq s/m^2$ then
\begin{equation} \label{skicsi}
|S+A|\leq |S||A|=\sqrt{|S|}\sqrt{|S|} |A| \leq \sqrt{|S|}
(\sqrt{s}/m ) |A|=\sqrt{s|S|}
\end{equation}
and we are done.

If $|S|> s/m^2$ then define $t=m-\sqrt{s/|S|}$, and use Theorem
\ref{pldiffnagy2} above to find a set $X\subset A$ such that
$|X|=r>t$ and \eqref{pld2} holds with $h=2$. For such an $X$ we
have
\begin{equation} \label{s1felso}
     |S+X|\leq |B_1+B_2+X| \leq
     \frac{s}{m-t} - {s\over m} + (|X|-t) {s\over (m-t)^2}
     \end{equation} and
\begin{equation} \label{s2felso}
|S+(A\setminus X)|\leq |S||A\setminus X|.
\end{equation}
We conclude that
\begin{equation} \label{sfelso}
|S+A|\leq |S+X|+|S+(A\setminus X)|\leq  \frac{s}{m-t} - {s\over m} +
(r-t) {s\over (m-t)^2}+
\end{equation}
\begin{equation*}
 +|S|\left ( (m-t)-(r-t)\right )=
2\sqrt{s|S|}-s/m\leq 2\sqrt{s|S|}.
\end{equation*}
This inequality is nearly the required one, except for the factor
of 2. We can
dispose of this factor as follows.  Consider the sets $A'=A^k,$
$B_1'=B_1^k,$ $B_2'=B_2^k$ and $S'=S^k$ in the $k$'th direct power of the
original group.
 Applying equation \eqref{sfelso}
to $A',$ etc., we obtain
\begin{equation} \label{s'felso}
|S'+A'|\leq 2\sqrt{s'|S'|}.
\end{equation}
Since
 $|S'+A'|=|S+A|^k,$ $s'=s^k$ and
$|S'|=|S|^k$,  we get
\begin{equation} \label{sfelsouj}
|S+A|\leq 2^{1/k}\sqrt{s|S|}.
\end{equation}
Taking the limit as $k\to \infty $ we obtain the desired
inequality
\begin{equation} \label{sfelsovege}
|S+A|\leq \sqrt{s|S|}.
\end{equation}

     \end{proof}

     {\bf Acknowledgement.} The authors are grateful to Vsevolod Lev for
several relevant comments quoted in the paper, and to Katalin
Marton for discussions on entropy connection and for directing our
attention to some relevant sources.

    % \bibliographystyle{amsplain}
     %\bibliography{cimek,cikkeim}

     \end{document}